\documentclass[a4paper,11pt]{amsart}

\usepackage[
  margin=30mm,
  marginparwidth=25mm,     
  marginparsep=2mm,       
  bottom=25mm,
  ]{geometry}

\usepackage[T1]{fontenc}
\usepackage[utf8]{inputenc}
\usepackage[english]{babel}
\usepackage{amsmath,amssymb,amsthm,mathtools}
\usepackage{latexsym}
\usepackage{delarray}
\usepackage{bbm}
\usepackage{scrextend}
\usepackage{hyperref}
\usepackage[pdftex,usenames,dvipsnames]{xcolor}
\usepackage{datetime}
\usepackage{mathrsfs}
\usepackage{enumitem}  
\usepackage{tikz}
\usepackage{comment}
\usepackage{multicol}

\newtheorem{Th}{Theorem}[section]
\newtheorem{Prop}[Th]{Proposition}

\newtheorem{Cor}[Th]{Corollary}

\DeclareMathOperator{\supp}{supp}
\DeclareMathOperator*{\esssup}{ess\,sup}

\newcommand{\R}{\mathbb{R}}

\newcommand{\C}{\mathbb{C}}

\usepackage{isomath}

\allowdisplaybreaks 

\title[The NLSA in weighted Sobolev spaces]
{The Nonlinear Schr\"odinger--Airy equation in weighted Sobolev spaces}

\author[A. J. Castro]{A. J. Castro}
\author[K. Jabbarkhanov]{K. Jabbarkhanov}
\author[L. Zhapsarbayeva]{L. Zhapsarbayeva}

\address{\newline
       Alejandro J. Castro, Khumoyun Jabbarkhanov, Lyailya Zhapsarbayeva \newline
       Department of  Mathematics, Nazarbayev University, \newline
		010000 Nur-Sultan, Kazakhstan}
\email{\{alejandro.castilla,khumoyun.jabbarkhanov\}@nu.edu.kz,
leylazhk67@gmail.com}

 \thanks{
A. J. C. and L. Z. are supported by the Nazarbayev University Faculty Development Competitive Research Grants Program 11022021FD2921.
K. J. is supported by the Ministry of Education and Science of the Republic of Kazakhstan via the project SSH2021016.}



\begin{document}

\begin{abstract}
We study the persistence property of the solution for the nonlinear Schr\"odinger--Airy equation with initial data in the weighted Sobolev space
$H^{1/4}(\R)\cap L^2(|x|^{2m}dx)$, $0<m \leq 1/8$, via the contraction principle. 
\end{abstract}

\maketitle



\section{Introduction}
In this paper we consider the initial value problem for the 
\textit{nonlinear Schr\"odinger--Airy equation} (NLSA)
\begin{equation}\label{eq:NLSAE}
\left\{
\begin{array}{ll}
\partial_t u 
+ i \, a \, \partial_x^2 u  
+ b \, \partial_x^3 u
+ i \, c \, |u|^2 u 
+ d \, |u|^2 \partial_x u
+ e \, u^2 \partial_x \overline{u} =0, & x \in \mathbb{R}, \, t>0, \\
u(x,0)=u_0(x), & x \in \mathbb{R},
\end{array}
\right.
\end{equation}
where $u=u(x,t)$ is a complex valued function and 
$a,b,c,d,e$ are parameters. This model was proposed in \cite{K,HK} to describe the propagation of signals in optic fibers.\\

G. Staffilani \cite{Sta} showed that the equation \eqref{eq:NLSAE} is locally (in time) well-possed
(i.e. there exists a unique stable solution) for 
initial data in the $L^2$-based Sobolev space $H^s(\mathbb{R})$, for any $s \geq 1/4$ provided that
$a \in \R$, $b \in \R \setminus \{0\}$
and $c,d,e \in \C$.
Later, X. Carvajal \cite{Car2006} established the global well-posedness in $H^s(\mathbb{R})$, for $s > 1/4$ and real coefficients satisfying $be \neq 0$ and $c=(d-e)a/(3b)$. Moreover, Carvajal  obtained in \cite{Car2013} that the above results are sharp, in the sense that the map data-solution is not uniformly continuous for any $u_0 \in H^s(\mathbb{R})$ when $s<1/4$, under the constrains $e=0$, $bd>0$ and $c=ad/(3b)$.\\

For particular choices of the parameters, the IVP \eqref{eq:NLSAE} reduces to well-known equations of great relevance in quantum mechanics and water waves. For example, taking $a=- 1$,  $b=d=e=0$ and $c \in \R$, one gets the cubic \textit{nonlinear Schr\"odinger equation} (NLS)
\begin{equation*} 
i \partial_t u + \partial_x^2 u + c |u|^2 u =0.
\end{equation*} 
Its local and global well-posedness in $H^s(\mathbb{R})$, $s \geq 0$, was proven by  Y. Tsutsumi \cite{Tsu} while for $s<0$ it is ill-posed (see \cite{CCT2003,KPV}). Now, for $a=c=e=0$, and $b=d=1$ we obtain the  
\textit{modified Korteweg-de Vries equation} (mKdV)
\begin{equation*} 
\partial_t u + \partial_x^3 u + |u|^2 \partial_x u =0.
\end{equation*}
C. Kenig, G. Ponce and L. Vega \cite{KPV1993} showed
that it is local well-posed in $H^s(\R)$, $s \geq 1/4$;  and  J. Colliander, M. Keel, G. Staffilani, H. Takaoka and T. Tao \cite{CKSTT} demonstrated that the real mKdV is globally well-posed for $s > 1/4$. Furthermore, these results are again optimal (\cite{CCT2003,KPV}). On the other hand,
if we fix $a=-1$, $b=c=0$ and $d=2e \in \R$ we have the
\textit{derivative nonlinear Schr\"odinger equation} (DNLS)
\begin{equation*}
i \partial_t u 
+ \partial_x^2 u 
+i e \partial_x(|u|^2u)
=0,
\end{equation*}
whose local well-posedness was provided by H. Takaoka \cite{Tak1999} for $s \geq 1/2$ and globally for $s>1/2$ in 
\cite{CKST2002}.\\

So far, all mentioned results concern with solvability in Sobolev spaces.  As T. Kato remarked in \cite{Kato1983}, functions in
$H^s(\R)$ do not necessarily decay fast at infinity,
which is in many applications unrealistic.
In order to show well-posedness in spaces of fast--decaying functions we can replace $H^s(\R)$ by a weighted Sobolev space of type
$H^s(\R) \cap L^2(w(x)dx)$,
for some appropriate $w$ which is typically taken of power type. 
The latest results in this direction for the NLS have been obtained by 
J. Nahas and G. Ponce \cite{NP2009},
for the mKdV by J. Nahas \cite{Nah2012}
(see also \cite{FonLinPn}) and for the NLSA by X. Carvajal and W. Neves \cite{CN2010}.  
It is also worth mentioning that over the past few years this weighted Sobolev approach has been carried out for several nonlinear evolution equations 
(see for instance 
\cite{BusJi,BusJi2016,BusJi2018,CN2015,CunPas,FonPn,ILP2013,ILP2015,Jim2013}, and references therein).\\

Coming back to the IVP for the NLSA, and by means of 
 Duhamel’s principle, it is straightforward to show that any solution of the partial differential equation \eqref{eq:NLSAE} also satisfies the integral equation
\begin{equation}\label{iEq01}
u=e^{-t(ia\partial^2_{x}+b\partial^3_{x})}u_0-\int^t_{0}e^{-(t-t')(ia\partial^2_{x}+b\partial^3_{x})}
\big(i \, c \, |u|^2 u 
+ d \, |u|^2 \partial_x u
+ e \, u^2 \partial_x \overline{u}\big)(\cdot,t')\, dt', 
\end{equation}
where $e^{-t(ia\partial^2_{x}+b\partial^3_{x})}$ stands for the Schr\"odinger--Airy semigroup (see Section \ref{Sect:SchoAirSem} below for its precise definition). However, notice that not all solutions of \eqref{iEq01} also verify \eqref{eq:NLSAE}, specially if $u_0$ belongs to a Sobolev space of low regularity (see for example \cite[p. 96]{LP2015}).\\

We are now in position of stating our main result.

\begin{Th}\label{Th:main} 
Let $a \in \R$, $b \in \R \setminus \{0\}$ and 
$c,d,e \in \mathbb{C}$ with $d=2e$.
Assume that $u_{0}\in H^{1/4}(\R)\cap L^2(|x|^{2m}dx)$ for $0<m \leq 1/8$.
Then, there exist $T>0$ and a unique 
 solution $u$ of the integral equation \eqref{iEq01} such that 
\begin{equation}\label{eq:persistenceu}
u(\cdot,t)\in H^{1/4}(\R)\cap L^2(|x|^{2m}dx),
\quad t \in (0,T].
\end{equation}
\end{Th}

\quad \\
A few remarks are in order. Firstly, Theorem \ref{Th:main} improves on the Sobolev regularity for the local part of \cite[Theorem 3.7]{CN2010} established for $s=2$. Recall that for the unweighted case ($m=0$), $s=1/4$ is  optimal for the NLSA (\cite{Car2013}).
Moreover, our restriction of $2m \leq s$ was shown to be sharp in the context of the mKdV (see \cite[Remark (b) on p. 5356]{FonLinPn}), which makes it natural to conjecture that in the NLSA setting is too.
Finally, it is worth-mentioning that our approach imposes $d=2e$ only in the analysis of the term $B_{2,2,2}$ at the very end of Section \ref{Sect:ProfTh1.1}, in order to exploit the symmetry given by \eqref{eq:fullderiv}
and hence take advantage of the property \eqref{nah3.3}. It is an open question for the authors how to adapt the reasoning for arbitrary values $d$ and $e$. \\

This paper is organized as follows. 
In Section \ref{Sect:Notation} we introduce notation as well as the main properties used later. 
Deeply inspired by the work of Nahas \cite{Nah2012} for the mKdV equation, in Section \ref{Sect:crucial}
we generalize his crucial estimate to the Schr\"odinger--Airy setting. Finally in Section \ref{Sect:ProfTh1.1} we prove Theorem \ref{Th:main} via the contraction principle.

\quad 

\section{Notation, definitions and Preliminaries}
\label{Sect:Notation}

Throughout the paper we use the standard notation $A \lesssim B$ when there exists $C>0$ such that $A \leq C B$, and $A \sim B$ when simultaneously $A \lesssim B$ and $B \lesssim A$. 

\subsection{Function spaces}
For $f : \R \rightarrow \C$
consider its Fourier transform 
$$ 
\widehat{f} (\xi)
:= 
\int_{\R} e^{-ix\xi} f(x) \, dx, 
\quad \xi\in \R
$$
and recall the inverse Fourier formula
$$ 
f(x)
:=  
\frac{1}{2\pi}
\int_{\R} e^{ix\xi} \widehat{f}(\xi) \, d\xi.
$$
We also introduce the $L^2$--based Sobolev space $H^s(\mathbb{R})$ of order $s \in \mathbb{R}$ via the norm
$$
\|f\|_{H^s}
:=
\Big(  \int_{\R}
\langle \xi \rangle^{2s}
|\widehat{f}(\xi)|^2 \, d\xi\Big)^{1/2}$$
for 
$\langle \xi \rangle
:=(1+|\xi|^2)^{1/2}$.
It is well-known that they satisfy the inclusion
$H^{s_2}(\R) \subset H^{s_1}(\R)$, $s_1 \leq s_2$, or equivalently
\begin{equation}\label{eq:Sobolevinclusion}
\|f\|_{H^{s_1}}
\lesssim
\|f\|_{H^{s_2}}, \quad s_1 \leq s_2.
\end{equation}

\quad \\
Now, for a function of two variables
 $f : \R \times [0,T] \rightarrow \C$
and $1 \leq p, q \leq \infty$, let's
define the mixed norms
$$
\|f\|_{L^p_xL^q_T}
:=\Big( 
\int_{\R}
\Big|\int^T_0 |f(x,t)|^q \, dt\Big|^{p/q} \, dx
\Big)^{1/p} 
$$
and
$$
\|f\|_{L^q_TL^p_x}
:=\Big( 
\int^T_0
\Big|\int_{\R} |f(x,t)|^p \, dx\Big|^{q/p} \, dt
\Big)^{1/q}
$$
with the usual modifications when $p=\infty$
or $q=\infty$. Moreover, we write $\widehat{f}$ meaning that the Fourier transform is taken with respect to the spatial variable $x$.\\

\subsection{Schr\"odinger--Airy semigroup}
\label{Sect:SchoAirSem}
We denote by 
$e^{-t(ia\partial^2_{x}+b\partial^3_{x})}$
to the free propagator of the linear Schr\"{o}dinger--Airy equation 
\begin{equation*} 
\left\{
\begin{array}{ll}
\partial_t u 
+ i \, a \, \partial_x^2 u  
+ b \, \partial_x^3 u
=0, & x \in \mathbb{R}, \, t>0, \\
u(x,0)=u_0(x), & x \in \mathbb{R},
\end{array}
\right.
\end{equation*}
whose solution 
$u(x,t)
=:e^{-t(ia\partial^2_{x}+b\partial^3_{x})}u_0(x)$ 
is easily represented on the Fourier side
\begin{equation}\label{eq:Schrsemig}
(e^{-t(ia\partial^2_{x}+b\partial^3_{x})} u_0)^{\wedge}(\xi)
= e^{it(a \xi^2+b \xi^3)} \widehat{u}_0(\xi).
\end{equation}
In \cite[Theorem 4.1(ii)]{Sta} the following estimate was established
\begin{equation}\label{nah3.3}
\Big\|
\partial_x \int^t_0 e^{-(t-t')(ia\partial^2_{x}+b\partial^3_{x})}f(\cdot,t')\, dt'
\Big\|_{L^\infty_T L^2_x}
\lesssim
\|f\|_{L^{1}_xL^{2}_T},
\end{equation}
which will be crucial in Section \ref{Sect:ProfTh1.1}.\\

\subsection{Fractional derivatives} 
For $\alpha\in\R$, it is standard to define the fractional derivative $D_x^{\alpha}$ as the Fourier multiplier given by
\begin{equation}\label{eq:fracderiv}
\widehat{D_x^{\alpha}f}(\xi)
:=
|\xi|^{\alpha}\widehat{f}(\xi).
\end{equation}
In this notes we will restrict ourselves to the case of $0 < \alpha \leq 1$.
Our motivation for considering these derivatives is in connection with the Sobolev spaces and also with the $L^2$--weighted spaces for non-integer powers $m$, since Plancherel's theorem provides
\begin{equation*}
\|f\|_{L^2(|x|^{2m},dx)}
:= 
\||x|^m f\|_{L^2_x}
\sim
\|D^m_{\xi} \widehat{f}\|_{L^2_{\xi}}.
\end{equation*}

For completeness, we state below the main properties involving fractional derivatives that will be needed in the sequel.\\

In \cite[Lemma 4.7]{Sta} it was shown that, for some $\theta>0$
\begin{equation}\label{Sta4.7}
\| f\|_{L^5_TL^{\infty}_x}
\lesssim 
T^{\theta}
\big(
\| f\|_{L^5_xL^{10}_T}
+
\|D^{1/4}_x f\|_{L^5_xL^{10}_T} \big).  
\end{equation}

\quad \\
Let's introduce the commutator 
$$
[\phi,D_x^{\alpha}]f
:= 
\phi D_x^{\alpha} f- D_x^{\alpha} ( \phi f).
$$
When $\alpha=1$, it corresponds to the so called first commutator of Calder\'on (see \cite{Calderon}).
Furthermore,  it is known that
(see \cite[Lemma 2.2]{Murray}) 
\begin{equation}\label{eq:commutator}
\|[\phi,D_x^{\alpha}] f\|_{L^2_{x}} 
\lesssim
\|\partial_x \phi \|_{L^{\infty}_x}
\| f\|_{L^2_{x}}.
\end{equation}

\quad \\
Now, we recall a chain rule for fractional derivatives 
(\cite[Theorem A.7]{KPV1993})
\begin{equation} \label{eq:chainrule}
\|D_x^\alpha F(f) \|_{L^p_x}
\lesssim 
\|F'(f)\|_{L^\infty_x}
\|D_x^\alpha(f) \|_{L^p_x},
\quad 1<p<\infty
\end{equation}
and its space-time analogous (\cite[Theorem A.6]{KPV1993})
\begin{equation}\label{eq:KPVThA6}
\|D_x^{\alpha} F(f)\|_{L^{p}_x L^{q}_T}
\lesssim
\|F'(f)\|_{L^{p_1}_x L^{q_1}_T}
\|D_x^{\alpha} f\|_{L^{p_2}_x L^{q_2}_T}
\end{equation}
for exponents $1 < p, p_1, p_2, q, q_2 < \infty$
and $1<q_1 \leq \infty$
satisfying 
\begin{equation}\label{eq:exponents}
\frac{1}{p} = \frac{1}{p_1} + \frac{1}{p_2}, \quad 
\frac{1}{q} = \frac{1}{q_1} + \frac{1}{q_2}.
\end{equation}

\quad \\
We finish this review with some fractional versions of the Leibniz rule.  For exponents as in \eqref{eq:exponents} one has that 
(see \cite[Theorem A.8]{KPV1993})
\begin{equation}\label{eq:KPVThA8}
\|
D_x^\alpha(fg)
- f D_x^\alpha g
- g D_x^\alpha f
\|_{L^p_x L^q_t}    
\lesssim 
\|D_x^{\alpha_1} f\|_{L^{p_1}_x L^{q_1}_t}
\|D_x^{\alpha_2} g\|_{L^{p_2}_x L^{q_2}_t}
\end{equation}
provided that $\alpha = \alpha_1 + \alpha_2$
with $0 \leq \alpha_1, \alpha_2 \leq \alpha$.
Moreover, the more delicate estimate
\begin{equation}\label{NLem3.2}
\|D^{\alpha}_x(fg)-gD^{\alpha}_xf\|_{L^2_x}
\lesssim 
\|Q_ND^{\alpha}_xg\|_{L_x^{\infty}\ell^1_N}
\|f\|_{L^2_x}
\end{equation}
was obtained in \cite[Lemma 4.2]{Nah2012},
where $Q_N$ represents the multiplier operator
\begin{equation*}
\widehat{Q_N f} (x)
:=
\big(
\eta( 2^{-N}x)
+
\eta( - 2^{-N}x)
\big)
\widehat{f}(x), \quad N \in \mathbb{Z}
\end{equation*}
and $\eta$ is a smooth function  supported in $[1/2,2]$ that verifies
\begin{equation}\label{eq:defeta}
\sum_{N \in \mathbb{Z}} 
\big(
\eta( 2^{-N} x)
+
\eta( - 2^{-N} x)
\big)
=1, \quad x \in \R \setminus \{0\}.
\end{equation}
Here, as usual, we understand
\begin{equation*}
\|Q_ND^{\alpha}_xg\|_{L_x^{\infty}\ell^1_N}
:= 
\esssup_{x \in \R}
\sum_{N \in \mathbb{Z}}
|Q_ND^{\alpha}_xg(x)|.
\end{equation*}


\quad 

\section{Crucial estimate}
\label{Sect:crucial}
In the proof of Theorem \ref{Th:main} in Section \ref{Sect:ProfTh1.1} we will often use property \eqref{NLem3.2} to control the $L^2$-norm of fractional derivatives acting on the product of functions, in particular repeatedly for  
\begin{equation*}
g(\xi)
:= \frac{e^{it(a\xi^2+b\xi^3)}}{\langle\xi\rangle^{2m}}.
\end{equation*}
Therefore, it becomes necessary to understand the behaviour of 
$\|Q_ND^{m}_\xi g\|_{L_{\xi}^{\infty}\ell^1_N}$, which is the  purpose of this section
(see Corollary \ref{NLem2.4} below).
Furthermore, given the definition of $Q_N$, we will show that it is enough to estimate  oscillatory integrals of type
\begin{equation}\label{eq:oscintgoal}
\int_{\R}\varphi_{\omega}(\xi-z)\frac{e^{it(a z^2+b z^3)}}{\langle z\rangle^{2m}}dz,    
\end{equation}
where $\varphi$ denotes a smooth function with compact Fourier support away from the origin and 
\begin{equation}\label{eq:defdilation}
\varphi_\omega(z)
:= \omega \varphi(\omega z), \quad \omega \neq 0.
\end{equation}

\quad \\
Inspired by \cite[p. 842]{Nah2012}
(where the particular case of $a=0$ and $b=1$ was considered), we say that a point $\xi \in \R$ is \textit{near} if 
\begin{equation}\label{eq:defnear}
\Big|\xi+\frac{a}{2b}\Big|^2
\leq
\frac{1}{100}\frac{\omega}{|b|t}
+\Big(\frac{a}{2b}\Big)^2,
\end{equation}
\textit{intermediate} when
\begin{equation}\label{eq:definterm}
\frac{1}{100}\frac{\omega}{|b|t}
+
\Big(\frac{a}{2b}\Big)^2
<
\Big|\xi+\frac{a}{2b}\Big|^2
\leq 
100 \frac{\omega}{|b|t}
+\Big(\frac{a}{2b}\Big)^2,
\end{equation}
and
\textit{far} provide that 
\begin{equation}\label{eq:deffar}
\Big|\xi+\frac{a}{2b}\Big|^2
> 
100 \frac{\omega}{|b|t}+\Big(\frac{a}{2b}\Big)^2.
\end{equation}

We analyze now the behaviour of \eqref{eq:oscintgoal}.

\begin{Prop}\label{nahLm2.3}
Let $a \in \R$, $b \in \R \setminus \{0\}$, $t>0$ and $\omega >1$ verifying that 
\begin{equation}\label{eq:nicehypo}
\frac{w}{|b|t} \geq 
\max\Big\{1, 10^4 \Big(\frac{a}{2b}\Big)^2\Big\}.
\end{equation}
Moreover, assume that $ 0 \leq m < 1$ and   
$\varphi$ is a smooth function with
$\supp \widehat{\varphi} \subset [1/2,2]$.
Then 
\begin{equation}\label{eq:key}
\Big|
\int_{\R}\varphi_{\omega}(\xi-z)\frac{e^{it(a z^2+b z^3)}}{\langle z\rangle^{2m}}dz
\Big|
\lesssim 
\left\{
\begin{array}{ll}
(1+t) \, \omega^{-m}, & 
\text{for} \,\, \xi \,\, \text{intermediate},\\
(1+t) \, \omega^{-1}, & 
\text{for} \,\, \xi \,\, \text{near or far}.
\end{array}
\right.
\end{equation}
\end{Prop}

\begin{proof}
We start recalling some useful properties of the function $\varphi_\omega$.
In \cite[Lemma 3.1]{Nah2012}, it was shown that it is an entire function satisfying,
for $z = x + i y \in \C$ and $\xi \neq x$, 
\begin{equation}\label{nahas_lemma_21}
|\varphi_{\omega}(\xi-z)|
\lesssim 
\left\{
\begin{array}{ll}
\dfrac{|e^{2\omega y}-e^{\omega y/2}|}{\omega^2 y|\xi-z|^2}, 
& y\neq 0,\\
& \\
\dfrac{1}{\omega(\xi-x)^2},
& y=0.
\end{array}
\right.
\end{equation}
Moreover, as in \cite[Corollary 3.1]{Nah2012}, for any $\xi \in \R$ one gets that
\begin{equation*} 
 \varphi_{\omega}(\xi-z)\frac{e^{it(a z^2+b z^3)}}{\langle z\rangle^{2m}}  
\end{equation*}
is analytic on
$\C \setminus (\{|\textrm{Im}\, z|\geq 1\} \cap \{\textrm{Re}\,z=0\}$). The following estimate
 will also be very helpful  (see \cite[Lemma 3.2]{Nah2012})
\begin{equation}\label{nahas_lemma_22}
\int_0^\pi
\frac{
e^{\beta \sin \theta} 
-
e^{\alpha \sin \theta}}{\sin \theta} \, d\theta
\lesssim
\Big(\pi\frac{\alpha}{\beta}-1\Big)
+
1
+
\frac{1}{\pi}\frac{\beta}{\alpha}e^{-\pi \alpha/\beta}, \quad \alpha < \beta < 0.
\end{equation}

\quad \\
For $\varepsilon>0$, conveniently chosen later, and taking into account \eqref{nahas_lemma_21}, we readily deduce 
\begin{align}\label{eq:easypart}
\Big|
\int_{\R \setminus [\xi - \varepsilon, \xi + \varepsilon]}
\varphi_{\omega}(\xi-z)\frac{e^{it(a z^2+b z^3)}}{\langle z\rangle^{2m}}dz \Big|
& \lesssim
\int_{\R \setminus [\xi - \varepsilon, \xi + \varepsilon]}
\frac{1}{\omega(\xi-x)^2}dx
\lesssim \frac{1}{\omega\varepsilon}.
\end{align}
It remains to analyze the integral
\begin{equation*} 
I
:=
\int_{\xi - \varepsilon}^{\xi + \varepsilon}
\varphi_{\omega}(\xi-z)
\frac{e^{it(a z^2+b z^3)}}{\langle z\rangle^{2m}} \, dz
\end{equation*} 
for which we distinguish the situations of $\xi$ intermediate, near or far. \\

\underline{\textit{Intermediate case}}.
In this situation it is, for example, enough to fix $\varepsilon$ 
such that 
\begin{equation*}
100 \varepsilon^2
:=
\frac{1}{100}\frac{\omega}{|b|t}.
\end{equation*} 
Hence,
\begin{align}\label{eq:lowerbound}
|I| 
& 
\leq
\|\varphi\|_{L^1(\R)} \, 
\| |z|^{-2m}\|_{L^\infty(\xi-\varepsilon,\xi+\varepsilon)}
\lesssim
(|\xi|-\varepsilon)^{-2m}
\lesssim
\varepsilon^{-2m}
\lesssim
t^m \omega^{-m},
\end{align}
because by the reverse triangle inequality,
 \eqref{eq:definterm} and \eqref{eq:nicehypo} we have that
\begin{align*}
|\xi|-\varepsilon
& =
\Big|\xi+\frac{a}{2b} -\frac{a}{2b}\Big|
-\varepsilon
\geq
\Big|\xi+\frac{a}{2b}\Big|
- \frac{|a|}{2|b|}
-\varepsilon
> 
8 \varepsilon.
\end{align*}
Therefore, combining \eqref{eq:easypart} and \eqref{eq:lowerbound} we deduce
\begin{equation*}
\Big|
\int_{\R}\varphi_{\omega}(\xi-z)\frac{e^{it(a z^2+b z^3)}}{\langle z\rangle^{2m}}dz
\Big|
\lesssim 
\omega^{-1} + t^m \omega^{-m}
\lesssim (1+t)\omega^{-m},
\end{equation*}
as we wanted.\\

\underline{\textit{Near case}}.
Notice that, as long as $\varepsilon<1$, the Cauchy-Goursat theorem guarantees that 
\begin{equation} \label{eq:CG1}
I
=
\int_{\Gamma_1}
\varphi_{\omega}(\xi-z)\frac{e^{it(a z^2+b z^3)}}{\langle z\rangle^{2m}}dz,
\end{equation} 
where the latter integral is taken along the contour
\begin{equation*}
\Gamma_1
:= 
\{
z= \xi  + \varepsilon e^{i \theta}, \, 
- \pi \leq \theta \leq 0\}.
\end{equation*}
Furthermore, \eqref{nahas_lemma_21} implies
\begin{align} \label{eq:ncgamma1}
& \Big|
\int_{\Gamma_1}
\varphi_{\omega}(\xi-z)\frac{e^{it(a z^2+b z^3)}}{\langle z\rangle^{2m}} \, dz
\Big| \nonumber \\
& \qquad  \lesssim
\int_0^{\pi}
\dfrac{|e^{-2\omega \varepsilon \sin \theta}
-e^{-\frac{\omega \varepsilon}{2} \sin \theta}|}{\omega^2 \varepsilon^2 \sin \theta}
\exp\Big\{
\textrm{Re}\Big[
it
\Big(a 
( \xi  + \varepsilon e^{-i \theta})^2+
b 
( \xi  + \varepsilon e^{-i \theta})^3\Big)
\Big]\Big\}
\, d\theta.
\end{align}
For convenience, let's introduce the notation
$
\xi_{a,b}
:=
\xi+a/(2b)
$. Hence,
\begin{align*}
& \textrm{Re}\Big[
it
\Big(a 
( \xi  + \varepsilon e^{-i \theta})^2+
b 
( \xi  + \varepsilon e^{-i \theta})^3\Big)
\Big]\\
& \qquad =
t b \varepsilon \sin \theta 
\Big[
3 \xi^2
+
\Big(\frac{2a}{b} + 6 \varepsilon \cos \theta \Big)\xi 
+
\frac{2a}{b}
\varepsilon
\cos \theta 
+
4 \varepsilon^2 \cos^2 \theta
-
\varepsilon^2 
\Big] \\
& \qquad =
t b \varepsilon \sin \theta 
\Big[
2
\Big( 
\xi_{a,b} + \frac{3}{2} \varepsilon \cos \theta 
\Big)^2
+
\Big( 
\xi_{a,b} - \frac{a}{2b}\Big)^2
+
\Big( 
\frac{a}{2b} - \varepsilon \cos \theta \Big)^2 \\
& \qquad \qquad \qquad \quad 
-
3 \Big(\frac{a}{2b}\Big)^2
-
\frac{3}{2} \varepsilon^2 \cos^2 \theta
-
\varepsilon^2 
\Big].
\end{align*}
Thus, taking for example $\varepsilon:= 10^{-1}$, by using
\eqref{eq:defnear} and \eqref{eq:nicehypo} we can estimate
\begin{align*}
& \Big|
\textrm{Re}\Big[
it
\Big(a 
( \xi  + \varepsilon e^{-i \theta})^2+
b 
( \xi  + \varepsilon e^{-i \theta})^3\Big)
\Big]
\Big| \\   
& \quad \leq 
t|b| \varepsilon \sin \theta 
\Big[
2
\Big( 
|\xi_{a,b}| + \frac{3}{2}10^{-1}
\Big)^2
+
\Big( 
|\xi_{a,b}| + \frac{|a|}{2|b|}\Big)^2
+
\Big( 
\frac{|a|}{2|b|} + 10^{-1} \Big)^2 
+
3 \Big(\frac{a}{2b}\Big)^2
+
\frac{5}{2} 10^{-2}
\Big] \\
& \quad \leq 
t|b| \varepsilon \sin \theta 
\Big[
4
\Big( 
|\xi_{a,b}|^2 + \frac{9}{4}10^{-2}
\Big)
+
2\Big( 
|\xi_{a,b}|^2 + \Big(\frac{a}{2b}\Big)^2\Big)
+
2\Big( 
\Big(\frac{a}{2b}\Big)^2 + 10^{-2} \Big) \\
& \qquad \qquad \qquad 
+
3 \Big(\frac{a}{2b}\Big)^2
+
\frac{5}{2}10^{-2}
\Big] \\
& \quad \leq 
t|b| \varepsilon \sin \theta 
\Big[
6|\xi_{a,b}|^2 
+
7\Big(\frac{a}{2b}\Big)^2
+
14 \cdot 10^{-2}
\Big]
\leq 
\frac{\omega \varepsilon}{4} \sin \theta.
\end{align*}
Inserting the above bound in \eqref{eq:ncgamma1} and applying the property \eqref{nahas_lemma_22} we get
\begin{align}\label{eq:G1}
& \Big|
\int_{\Gamma_1}
\varphi_{\omega}(\xi-z)\frac{e^{it(a z^2+b z^3)}}{\langle z\rangle^{2m}} \, dz
\Big| 
\lesssim
\int_0^{\pi}
\dfrac{|e^{-2\omega \varepsilon \sin \theta}
-e^{-\frac{\omega \varepsilon}{2} \sin \theta}|}{\omega^2  \sin \theta}
e^{\frac{\omega \varepsilon}{4} \sin \theta}
\, d\theta \nonumber \\
& \qquad \lesssim
\omega^{-2}
\int_0^{\pi}
\dfrac{
e^{-\frac{\omega \varepsilon}{4} \sin \theta}
-
e^{-\frac{7 \omega \varepsilon}{4} \sin \theta}
}{\sin \theta}
\, d\theta
\lesssim
\omega^{-2}.
\end{align}
In conclusion, since $\omega>1$, putting together 
\eqref{eq:easypart}, \eqref{eq:CG1} and
\eqref{eq:G1}
we obtain \eqref{eq:key} for $\xi$ near.\\

\underline{\textit{Far case}}. 
This analysis is slightly different depending on whether $b$ is a positive or a negative number.
Take now 
$\varepsilon^2  
:=
\omega/(|b|t).
$
By the assumptions 
\eqref{eq:deffar} and \eqref{eq:nicehypo}
notice that
\begin{equation*}
|\xi \pm \varepsilon|
\geq
\Big| \xi + \frac{a}{2b} \Big|
-
\Big|\frac{a}{2b} \Big|
-
\varepsilon
>
(10-10^{-2}-1) 
\sqrt{\frac{\omega}{|b|t}}
> 8.
\end{equation*}
Hence, as in \eqref{eq:CG1}, applying again the 
Cauchy-Goursat theorem we write for $b<0$
\begin{equation*} 
I
=
\int_{\Gamma_1}
\varphi_{\omega}(\xi-z)\frac{e^{it(a z^2+b z^3)}}{\langle z\rangle^{2m}}dz,
\end{equation*} 
while for $b>0$
\begin{equation*} 
I
=
- \int_{\Gamma_2}
\varphi_{\omega}(\xi-z)\frac{e^{it(a z^2+b z^3)}}{\langle z\rangle^{2m}}dz
\end{equation*} 
with 
\begin{equation*}
\Gamma_2
:= 
\{
z= \xi  + \varepsilon e^{i \theta}, \, 
0 \leq \theta \leq \pi\}.
\end{equation*}
Proceeding as in the near case above we get 
\begin{align*}
& \Big|
\int_{\Gamma_1}
\varphi_{\omega}(\xi-z)\frac{e^{it(a z^2+b z^3)}}{\langle z\rangle^{2m}} \, dz
\Big| \nonumber \\
& \qquad  \lesssim
t \omega^{-3}
\int_0^{\pi}
\dfrac{
e^{-\frac{\omega \varepsilon}{2} \sin \theta}
-
e^{-2\omega \varepsilon \sin \theta}
}{\sin \theta}
\exp\Big\{
\textrm{Re}\Big[
it
\Big(a 
( \xi  + \varepsilon e^{-i \theta})^2+
b 
( \xi  + \varepsilon e^{-i \theta})^3\Big)
\Big]\Big\}
\, d\theta
\end{align*}
and also
\begin{align*} 
& \Big|
\int_{\Gamma_2}
\varphi_{\omega}(\xi-z)\frac{e^{it(a z^2+b z^3)}}{\langle z\rangle^{2m}} \, dz
\Big|  \\
& \qquad  
\lesssim
t w^{-3}
\int_0^{\pi}
\dfrac{e^{2\omega \varepsilon \sin \theta}
-e^{\frac{\omega \varepsilon}{2} \sin \theta}}{ \sin \theta}
\exp\Big\{
\textrm{Re}\Big[
it
\Big(a 
( \xi  + \varepsilon e^{i \theta})^2+
b 
( \xi  + \varepsilon e^{i \theta})^3\Big)
\Big]\Big\}
\, d\theta,
\end{align*}
where
\begin{align*}
& \textrm{Re}\Big[
it
\Big(a 
( \xi  + \varepsilon e^{\mp i \theta})^2+
b 
( \xi  + \varepsilon e^{\mp i \theta})^3\Big)
\Big]\\
& \qquad =
\pm t b \varepsilon \sin \theta 
\Big[
2
\Big( 
\xi_{a,b} + \frac{3}{2} \varepsilon \cos \theta 
\Big)^2
+
\Big( 
\xi_{a,b} - \frac{a}{2b}\Big)^2
+
\Big( 
\frac{a}{2b} - \varepsilon \cos \theta \Big)^2 \\
& \qquad \qquad \qquad \qquad 
-
3 \Big(\frac{a}{2b}\Big)^2
-
\frac{3}{2} \varepsilon^2 \cos^2 \theta
-
\varepsilon^2 
\Big].
\end{align*}
Moreover, \eqref{eq:deffar} and \eqref{eq:nicehypo} provide
\begin{align*}
& 
\Big[2
\Big( 
\xi_{a,b} + \frac{3}{2} \varepsilon \cos \theta 
\Big)^2
+
\Big( 
\xi_{a,b} - \frac{a}{2b}\Big)^2
+
\Big( 
\frac{a}{2b} - \varepsilon \cos \theta \Big)^2 
-
3 \Big(\frac{a}{2b}\Big)^2
-
\frac{3}{2} \varepsilon^2 \cos^2 \theta
-
\varepsilon^2 \Big] \\
& \qquad 
>
\Big( 
|\xi_{a,b}| - \Big|\frac{a}{2b}\Big|\Big)^2
-
3 \Big(\frac{a}{2b}\Big)^2
-
\frac{5}{2} \varepsilon^2 
>
\Big[
\Big( 10 - 10^{-2} \Big)^2 
-
3 \cdot 10^{-4}
-\frac{5}{2} 
\Big] \varepsilon^2 \\
& \qquad 
> 3 \frac{\omega}{|b|t}.
\end{align*}
Therefore, \eqref{nahas_lemma_22} implies 
for $b<0$
\begin{align*} 
& \Big|
\int_{\Gamma_1}
\varphi_{\omega}(\xi-z)\frac{e^{it(a z^2+b z^3)}}{\langle z\rangle^{2m}} \, dz
\Big|  
\lesssim
t w^{-3}
\int_0^{\pi}
\dfrac{
e^{\frac{-\omega \varepsilon}{2} \sin \theta}
-
e^{-2\omega \varepsilon \sin \theta}
}{ \sin \theta}
e^{- 3 \omega \varepsilon \sin \theta}
\, d\theta \\
& \qquad  \leq
t w^{-3}
\int_0^{\pi}
\dfrac{
e^{\frac{-\omega \varepsilon}{2} \sin \theta}
-
e^{-2\omega \varepsilon \sin \theta}
}{ \sin \theta}
\, d\theta
\lesssim t w^{-3},
\end{align*}
while for $b>0$,
\begin{align*} 
& \Big|
\int_{\Gamma_2}
\varphi_{\omega}(\xi-z)\frac{e^{it(a z^2+b z^3)}}{\langle z\rangle^{2m}} \, dz
\Big|  
\lesssim
t w^{-3}
\int_0^{\pi}
\dfrac{e^{2\omega \varepsilon \sin \theta}
-e^{\frac{\omega \varepsilon}{2} \sin \theta}}{ \sin \theta}
e^{- 3 \omega \varepsilon \sin \theta}
\, d\theta \\
& \qquad  =
t w^{-3}
\int_0^{\pi}
\dfrac{
e^{-\omega \varepsilon \sin \theta}
-
e^{-\frac{ 5 \omega \varepsilon}{2} \sin \theta}
}{ \sin \theta}
\, d\theta
\lesssim t w^{-3},
\end{align*}
which, in combination with \eqref{eq:easypart}, give \eqref{eq:key} for $\xi$ far.
\end{proof}

As a consequence of Proposition \ref{nahLm2.3} we obtain the following.

\begin{Cor}\label{NLem2.4}
Let $a \in \R$, $b \in \R \setminus \{0\}$, $t>0$ and 
$0 < m < 1$. We have that
\begin{equation*}
\Big\|
Q_N D^{m}_{\xi}\Big(\frac{e^{it(a\xi^2+b\xi^3)}}{\langle\xi\rangle^{2m}}\Big)
\Big\|_{L_{\xi}^{\infty}\ell^1_N}
\lesssim 1+t.   
\end{equation*}
\end{Cor}

\begin{proof}
For $\eta$ as in \eqref{eq:defeta},
we start introducing a new Fourier multiplier $Q_N^m$
via
\begin{equation*}
\widehat{Q_N^m f} (x)
:=
| 2^{-N} x |^m
\big(
\eta( 2^{-N}x)
+
\eta(- 2^{-N}x) \big)
\widehat{f}(x)
\end{equation*}
which, together with \eqref{eq:fracderiv}, allows us to write
\begin{equation*}
Q_N D^{m}_{\xi} f
=
2^{Nm} Q_N^m f.
\end{equation*}
Notice that
\begin{align*} 
|Q_N^m f (\xi)| 
& 
\lesssim 
\Big|
\int_{\R}
h(\xi-z)f(z) \, dz 
\Big|
+
\Big|
\int_{\R}
h(-\xi-z)f(-z) \, dz 
\Big|
\end{align*}
for
$\widehat{h}(x)
:= 
| 2^{-N} x |^m
\eta( 2^{-N}x )$.
Moreover, if we further call
$\widehat{\varphi}(x):=|x|^m \eta(x)$, it is clear that 
$\supp \widehat{\varphi} \subset [1/2,2]$
and also $\varphi_{2^N}=h$ because, by the convention \eqref{eq:defdilation}, we see that
\begin{align*}
\widehat{\varphi_{2^N}}(x)  
=
2^N \widehat{\varphi (2^N \cdot)}(x)  
=
\widehat{\varphi} (2^{-N}x) 
=
\widehat{h}(x).
\end{align*}
Next, we choose $N_0 \in \mathbb{N}$ such that 
\begin{equation*}
2^{N_0} \geq 
|b|t \max\Big\{1, 10^4 \Big(\frac{a}{2b}\Big)^2\Big\}. 
\end{equation*}
Thus, 
\begin{align*}
& \Big\|
\sum_{N=-\infty}^{N_0}
\Big|Q_N D^{m}_{\xi}\Big(\frac{e^{it(a\xi^2+b\xi^3)}}{\langle\xi\rangle^{2m}}\Big)\Big|
\Big\|_{L_{\xi}^{\infty}}    
\leq
\sum_{N=-\infty}^{N_0}
2^{Nm}
\Big\|
Q_N^m \Big(\frac{e^{it(a\xi^2+b\xi^3)}}{\langle\xi\rangle^{2m}}\Big)
\Big\|_{L_{\xi}^{\infty}} \\
& \qquad \lesssim
\sum_{N=-\infty}^{N_0}
2^{Nm} \|\varphi_{2^N}\|_{L^1(\R)}
\Big\|\frac{e^{it(a z^2 \pm b z^3)}}{\langle z\rangle^{2m}} \Big\|_{L^\infty(\R)}
\lesssim
\sum_{N=-\infty}^{N_0}
2^{Nm}
\lesssim 1.
\end{align*}
On the other hand, Proposition \ref{nahLm2.3} provides
\begin{align*}
& \Big\|
\sum_{N=N_0+1}^{\infty}
\Big|Q_N D^{m}_{\xi}\Big(\frac{e^{it(a\xi^2+b\xi^3)}}{\langle\xi\rangle^{2m}}\Big)\Big|
\Big\|_{L_{\xi}^{\infty}} \\
& \qquad \lesssim 
\Big\|
\sum_{\substack{N=N_0+1\\ \xi \textrm{ intermediate}}}^{\infty}
2^{Nm} \Big| {Q}_N^m \Big(\frac{e^{it(a\xi^2+b\xi^3)}}{\langle \xi\rangle^{2m}}\Big)\Big|
\Big\|_{L_{\xi}^{\infty}} 
+
\Big\|
\sum_{\substack{N=N_0+1\\ \xi \textrm{ near or far}}}^{\infty}
2^{Nm} \Big| {Q}_N^m \Big(\frac{e^{it(a\xi^2+b\xi^3)}}{\langle \xi\rangle^{2m}}\Big)\Big|
\Big\|_{L_{\xi}^{\infty}}\\ 
& \qquad  
\lesssim 
(1+t)
\Big(
\Big\|
\sum_{\substack{N=N_0+1\\ \xi \textrm{ intermediate}}}^{\infty} 2^{Nm} 2^{-Nm}
\Big\|_{L_{\xi}^{\infty}}
+ \sum^{\infty}_{N=N_0+1} 2^{Nm} 2^{-N}\Big)
\lesssim 
1+t,
\end{align*}
because there are finitely many $N$ for which  a fixed $\xi$ is intermediate, recall \eqref{eq:definterm} taken with $\omega=2^N$. 
\end{proof}

\section{Proof of Theorem \ref{Th:main}}
\label{Sect:ProfTh1.1}
Consider the mapping
\begin{equation*}
\Phi(u):=e^{-t(ia\partial^2_{x}+b\partial^3_{x})}u_0-\int^t_{0}e^{-(t-t')(ia\partial^2_{x}+b\partial^3_{x})}(i \, c \, |u|^2 u 
+ d \, |u|^2 \partial_x u
+ e \, u^2 \partial_x \overline{u})\, dt'
\end{equation*}
and for some $\rho,\, T>0$, that will be conveniently fixed later, define the complete metric space 
$$
X^{\rho}_T:=\{u \, : \, \|u\|_{X_T} \leq \rho \} 
$$
endowed with the norm
\begin{equation*} 
\|u\|_{X_T}:= \|u\|_{Y_T} +\||x|^{m}u\|_{L^{\infty}_TL^2_x},    
\end{equation*}
where
\begin{equation*}
\|u\|_{Y_T}
:= \mu^T_{1}(u)+\mu^T_{2}(u)+\mu^T_{3}(u)+\mu^T_{4}(u) +\mu^T_{5}(u),
\end{equation*}
and
\begin{align*}    
\mu^T_{1}(u)
& :=
\|u\|_{L_T^{\infty}L_x^2}
+
\|D_x^{1/4}u\|_{L_T^{\infty}L_x^2},\\  
\mu^T_{2}(u)
& :=
\|\partial_x u\|_{L_x^{\infty}L_T^2}
+
\|D_x^{1/4}\partial_xu\|_{L_x^{\infty}L_T^2},\\  
\mu^T_{3}(u)
&:=
\|\partial_x u\|_{L_x^{20}L_T^{5/2}},  \\
\mu^T_{4}(u)
&:= 
\|u \|_{L_x^{5}L_T^{10}}
+\|D_x^{1/4}u\|_{L_x^{5}L_T^{10}},\\  
\mu^T_{5}(u)
& := \|u\|_{L_x^{4}L_T^{\infty}}. 
\end{align*}
Our aim is to show that $\Phi$ is a contraction on $X^{\rho}_T$, and hence the Banach fixed-point theorem will guarantee the existence and uniqueness of $u \in X^{\rho}_T$ with $u=\Phi(u)$, which in particular verifies
\eqref{eq:persistenceu}. This will be achieved in the two steps below.\\

\underline{\textit{Step 1: $\Phi$ is well defined}}.
The goal is to see that $\Phi$ maps $X^{\rho}_T$ into itself, that is, for $u \in X^{\rho}_T$ we want to show that
\begin{equation}\label{eq:goalNLSA}
\|\Phi(u)\|_{X_T}\leq \rho.
\end{equation}
Firstly, recall the well-known result for the unweighted case $m=0$ (see \cite[Proof of Theorem 3.3]{Sta})  
\begin{equation}\label{eq:goalPHI'}
\|\Phi(u)
\|_{Y_{T}}
\lesssim 
\|u_0\|_{H^{1/4}}+
T^{\theta}\| u \|_{Y_T}^3.
\end{equation}
Along this proof, with a slight abuse of notation, $\theta$ represents a positive number that might change from one line to another but its particular value it is not relevant for the argument.
Hence, it remains to control the terms
\begin{align*}
\||x|^{m}\Phi(u)\|_{L^{\infty}_TL^2_x}
&\leq 
\||x|^{m}e^{-t(ia\partial^2_{x}+b\partial^3_{x})}u_0\|_{L^{\infty}_TL^2_x}\nonumber
\\
& \quad + 
\Big\|
|x|^{m}\int^t_{0}e^{-(t-t')(ia\partial^2_{x}+b\partial^3_{x})}(i \, c \, |u|^2 u 
+ d \, |u|^2 \partial_x u
+ e \, u^2 \partial_x \overline{u})\, dt'
\Big\|_{L^{\infty}_TL^2_x}\nonumber
\\ \label{iEq1w}
& =: L+NL. 
\end{align*}

\quad \\
We start analyzing the linear factor $L$. 
Plancherel's formula, \eqref{eq:fracderiv} and \eqref{eq:Schrsemig} allow us to write 
\begin{align*}
L
&= 
\|D^{m}_{\xi}(e^{it(a\xi^2+b\xi^3)}\widehat{u}_0)\|_{L^{\infty}_TL^2_{\xi}} \\
& \lesssim 
\Big\|
D^{m}_{\xi}\Big(\frac{e^{it(a\xi^2+b\xi^3)}}{\langle\xi\rangle^{2m}}\langle\xi\rangle^{2m}\widehat{u}_0\Big)
-
\frac{e^{it(a\xi^2+b\xi^3)}}{\langle\xi\rangle^{2m}}D^{m}_{\xi}\Big(\langle\xi\rangle^{2m}\widehat{u}_0\Big)
\Big\|_{L^{\infty}_TL^2_{\xi}}
\\
& \quad +
\Big\|
\frac{e^{it(a\xi^2+b\xi^3)}}{\langle\xi\rangle^{2m}}D^{m}_{\xi}\Big(\langle\xi\rangle^{2m}\widehat{u}_0\Big)
\Big\|_{L^{\infty}_TL^2_{\xi}}\\
& =: L_1+L_2.
\end{align*}
For the first term above we use inequality \eqref{NLem3.2} together with Corollary \ref{NLem2.4} 
and \eqref{eq:Sobolevinclusion} to deduce
\begin{equation}\label{eq:L1}
L_1
\lesssim 
\Big\|
Q_N D^{m}_{\xi}\Big(\frac{e^{it(a\xi^2+b\xi^3)}}{\langle \xi\rangle^{2m}}\Big)
\Big\|_{L^\infty_T L_\xi^{\infty}\ell^1_N}
\|u_0\|_{H^{2m}}
\lesssim 
(1+T)\|u_0\|_{H^{1/4}}.   
\end{equation}
As for $L_2$, we apply the commutator estimate \eqref{eq:commutator} and get
\begin{align*}
L_2 
& \leq 
\|[\langle\xi\rangle^{-2m},D^{m}_{\xi}](\langle\xi\rangle^{2m}\widehat{u}_0)\|_{L^2_{\xi}}+\|D^{m}_{\xi}\widehat{u}_0\|_{L^2_{\xi}}  
\lesssim 
\|u_0\|_{H^{1/4}}
+
\| |x|^{m}u_0\|_{L^2}. 
\end{align*}
Hence,
\begin{equation}\label{eq:Lsummary}
L 
\lesssim 
(1+T)\|u_0\|_{H^{1/4}}
+
\| |x|^{m}u_0\|_{L^2}.
\end{equation}

\quad \\
Next, we focus on the nonlinear factor $NL$. 
For convenience, we decompose it as follows
\begin{align}
 NL 
 & \lesssim 
 \Big\|
 D^{m}_{\xi}\Big(\int^t_{0}
 e^{i(t-t')(a \xi^2+b \xi^3)}(|u|^2 u )^{\wedge}(\xi,t')\, dt'\Big)
 \Big\|_{L_T^{\infty}L_{\xi}^{2}} \nonumber
\\
& \qquad  
+\Big\|D^{m}_{\xi}\Big(\int^t_{0}e^{i(t-t')(a \xi^2+b \xi^3)}(  d \, |u|^2 \partial_x u
+ e \, u^2 \partial_x \overline{u})^{\wedge }(\xi,t')\, dt'\Big)\Big\|_{L_T^{\infty}L_{\xi}^{2}}\nonumber  
\\ \label{A+B}
&  =: A+B.
\end{align}

\quad \\
\noindent
We can further write   
\begin{align*}
A 
&   \leq 
\Big\|D^{m}_{\xi}\Big(\int^t_{0}\frac{e^{i(t-t')(a \xi^2+b \xi^3)}}{\langle\xi\rangle^{2m}}\langle\xi\rangle^{2m}(|u|^2 u)^{\wedge}(\xi,t')\, dt'\Big)\\
& \qquad \qquad -
\int^t_{0}\frac{e^{i(t-t')(a \xi^2+b \xi^3)}}{\langle\xi\rangle^{2m}}D^{m}_{\xi}\Big(\langle\xi\rangle^{2m}(|u|^2 u)^{\wedge}(\xi,t')\Big)\, dt'\Big\|_{L_T^{\infty}L_{\xi}^{2}}\\
& \qquad  +  \Big\|\int^t_{0}\frac{e^{i(t-t')(a \xi^2+b \xi^3)}}{\langle\xi\rangle^{2m}}D^{m}_{\xi}\Big(\langle\xi\rangle^{2m}(|u|^2 u)^{\wedge}(\xi,t')\Big)\, dt'\Big\|_{L_T^{\infty}L_{\xi}^{2}}\\
& = : A_1+A_2.
\end{align*}
Proceeding as in 
\eqref{eq:L1} and using also the H\"older and the Minkowski inequalities yield to
\begin{align}\label{A1}
A_{1}
&\lesssim 
(1+T)
\||u|^2 u\|_{L^1_T H^{2m}_x} 
\leq (1+T) T^{1/2}
\||u|^2 u\|_{L^2_T H^{1/4}_x} \nonumber\\
&   \lesssim 
T^\theta
\||u|^2u\|_{L^2_T L^2_x}
+
T^\theta
\|D_x^{1/4}(|u|^2u)\|_{L^2_TL^2_x}\nonumber\\
&  
=:  A_{1,1} + A_{1,2}.
\end{align}
Using  \eqref{Sta4.7} we obtain, 
\begin{align}\label{u2u}
A_{1,1}
& =
T^\theta
\Big(\int^T_{0} \int_{\R} |u|^6 \, dx \, dt\Big)^{1/2}
\lesssim 
T^\theta
\Big(
\int^T_{0} 
\|u\|^4_{L_x^{\infty}}
\|u\|^2_{L_x^2} \, dt
\Big)^{1/2}\nonumber\\
&  \lesssim 
T^\theta
\|u\|_{L^{\infty}_TL^2_x}
\|u\|^2_{L^5_TL^{\infty}_x}
\lesssim
T^\theta
\mu^T_{1}(u) 
\mu^T_{4}(u)^2,
\end{align}
and moreover an application of \eqref{eq:chainrule} similarly implies
\begin{align*} 
A_{1,2} 
& \lesssim  
T^\theta
\Big( \int^T_{0}\|D_x^{1/4}
(|u|^2u)\|^2_{L^2_x}\, dt\Big)^{1/2} 
\lesssim 
T^\theta
\Big( \int^T_{0}
\|u\|^4_{L^{\infty}_x}
\|D_x^{1/4}u\|_{L^2_x}^2\, dt\Big)^{1/2} \nonumber \\
& \lesssim 
T^\theta
\|D_x^{1/4}u\|_{L^{\infty}_TL^2_x}
\|u\|^2_{L^5_TL^{\infty}_x}
\lesssim 
T^\theta 
\mu^T_{1}(u) \mu^T_{4}(u)^2.
\end{align*}
In order to analyze  $A_2$ it is helpful to introduce the commutator below and invoke property \eqref{eq:commutator},
\begin{align}\label{eq:A2}
A_2 
& \leq
\Big\|
\int^t_{0}
e^{i(t-t')(a \xi^2+b \xi^3)}
[
\langle\xi\rangle^{-2m}, D^{m}_{\xi}
]
\Big(\langle\xi\rangle^{2m}(|u|^2 u)^{\wedge}(\xi,t')\Big)\, dt'
\Big\|_{L_T^{\infty}L_{\xi}^{2}} \nonumber \\ 
& \quad  
+
\Big\|\int^t_{0}e^{i(t-t')(a \xi^2+b \xi^3)}D^{m}_{\xi}\Big((|u|^2 u)^{\wedge}(\xi,t')\Big)\, dt'\Big\|_{L_T^{\infty}L_{\xi}^{2}} \nonumber\\
& \leq 
\Big\|
[\langle\xi\rangle^{-2m}, D^{m}_{\xi}]
\Big(\langle\xi\rangle^{2m}(|u|^2 u)^{\wedge}(\xi,t')\Big)\Big\|_{L_T^{1}L_{\xi}^{2}} \nonumber\\
& \quad 
+
\Big\|\int^t_{0}e^{i(t-t')(a \xi^2+b \xi^3)}D^{m}_{\xi}\Big((|u|^2 u)^{\wedge}(\xi,t')\Big)\, dt'\Big\|_{L_T^{\infty}L_{\xi}^{2}} \nonumber \\
& \lesssim 
\||u|^2 u\|_{L_T^{1}H_{x}^{2m}}
+
\Big\|\int^t_{0}e^{-(t-t')(ia\partial^2_{x}+b\partial^3_{x})}|x|^{m}
|u|^2 u \, dt'\Big\|_{L_T^{\infty}L_{x}^{2}} \nonumber\\ 
&   =:  A_{2,1}+A_{2,2}.
\end{align}
Notice that the first inequality of \eqref{A1} shows that $A_{2,1}$ can be treated as $A_1$. 
On the other hand, for $A_{2,2}$ we apply 
Minkowski's integral inequality, Plancherel's theorem, H\"older's inequality
and \eqref{Sta4.7} to arrive at
\begin{align*} 
A_{2,2} 
& \lesssim  
\| |u|^2 |x|^{m} u \|_{L^1_TL^2_x}   \leq
\|u\|_{L^4_T L^\infty_x}^{2} \||x|^{m} u \|_{L^2_TL^2_x}  
\leq 
T^\theta 
\|u\|^{2}_{L^5_T L^\infty_x}
\||x|^{m} u \|_{L^{\infty}_TL^2_x} \nonumber \\
& \lesssim
T^\theta
\mu^T_{4}(u)^2
\||x|^{m} u \|_{L^{\infty}_TL^2_x}.
\end{align*}
Combining above estimates we deduce
\begin{equation}\label{eq:Asummary}
A
\lesssim
T^\theta
\mu^T_{4}(u)^2
\Big(
\mu^T_{1}(u) 
+
\||x|^{m} u \|_{L^{\infty}_TL^2_x}
\Big).
\end{equation}

\quad \\
We continue with the analysis of the $B$ term in \eqref{A+B} and as usual decompose
\begin{align*}
B 
&  \leq
\Big\|D^{m}_{\xi}\Big(\int^t_{0}\frac{e^{i(t-t')(a \xi^2+b \xi^3)}}{\langle\xi\rangle^{2m}}\big(\langle\xi\rangle^{2m}(  d \, |u|^2 \partial_x u
+ e \, u^2 \partial_x \overline{u})^{\wedge}(\xi,t')\big)\, dt'\Big)\\
& \qquad -
\int^t_{0}\frac{e^{i(t-t')(a \xi^2+b \xi^3)}}{\langle\xi\rangle^{2m}}D^{m}_{\xi}\Big(\langle\xi\rangle^{2m}( d \, |u|^2 \partial_x u
+ e \, u^2 \partial_x \overline{u})^{\wedge}(\xi,t')\Big)\, dt'\Big\|_{L_T^{\infty}L_{\xi}^{2}}\\
& \quad + 
\Big\|\int^t_{0}\frac{e^{i(t-t')(a \xi^2+b \xi^3)}}{\langle\xi\rangle^{2m}}D^{m}_{\xi}\Big(\langle\xi\rangle^{2m}(d\,|u|^2\partial_x u
+ e\,u^2\partial_x\overline{u})^{\wedge}(\xi,t')\Big)\, dt'\Big\|_{L_T^{\infty}L_{\xi}^{2}}
\\
& =:B_1+B_2.
\end{align*}
The same arguments presented in \eqref{A1} to treat $A_1$, provide now
\begin{align*} 
B_{1} 
& \lesssim 
T^\theta 
\Big(
\|u^2\partial_x u\|_{L^2_x L^2_T}
+
\|D_x^{1/4}
(|u|^2 \partial_x u)\|_{L^2_x L^2_T}
+
\|D_x^{1/4}
(u^2 \partial_x \overline{u})\|_{L^2_x L^2_T} 
\Big) \nonumber\\
& =: B_{1,1} + B_{1,2} + B_{1,3}.
\end{align*}
Next, notice that
\begin{equation}\label{eq:B11}
B_{1,1}
\leq
T^\theta 
\|u\|_{L^4_x L^\infty_T}^2
\|\partial_x u\|_{L^\infty_x L^2_T}
\leq 
T^\theta 
\mu_5^T(u)^2
\mu_2^T(u).
\end{equation}
The analysis of $B_{1,2}$ and $B_{1,3}$ is similar, so we only discuss one case. Using 
\eqref{eq:KPVThA8} and \eqref{eq:KPVThA6} we get
\begin{align*}
B_{1,3}
& \lesssim
T^\theta
\Big(
\|
u^2
D_x^{1/4} \partial_x u
\|_{L^2_x L^2_T}
+
\|
\partial_x u
D_x^{1/4}u^2 
\|_{L^2_x L^2_T}
+
\|\partial_x u\|_{L^{20}_xL^{5/2}_T}
\|D_x^{1/4} u^2 \|_{L^{20/9}_xL^{10}_{T}}
\Big) \\
& \lesssim
T^\theta
\Big(
\|u^2\|_{L^2_x L^\infty_T}
\|D_x^{1/4} \partial_x u\|_{L^\infty_x L^2_T}
+
\|\partial_x u\|_{L^{20}_xL^{5/2}_T}
\|D_x^{1/4} u^2 \|_{L^{20/9}_xL^{10}_{T}}
\Big) \\
& \lesssim
T^\theta
\Big(
\|u\|_{L^4_x L^\infty_T}^2
\|D_x^{1/4} \partial_x u\|_{L^\infty_x L^2_T}
+
\|\partial_x u\|_{L^{20}_xL^{5/2}_T}
\|u\|_{L^{4}_xL^{\infty}_{T}}
\|D_x^{1/4} u \|_{L^{5}_xL^{10}_{T}}
\Big) \\
& \leq
T^\theta
\Big(
\mu_5^T(u)^2
\mu_2^T(u)
+
\mu_3^T(u)
\mu_5^T(u)
\mu_4^T(u)
\Big).
\end{align*}

For the treatment of $B_2$ we start as in  \eqref{eq:A2}, 
\begin{align*}
B_2
&  \lesssim 
T^{1/2}
\|
d\,|u|^2\partial_x u
+ e\,u^2\partial_x\overline{u}
\|_{L_T^{2}H_{x}^{1/4}} \\
& \quad +
\Big\|
\int^t_{0}
e^{-(t-t')(ia\partial^2_{x}+b\partial^3_{x})}
|x|^{m}
(d\,|u|^2\partial_x u
+ e\,u^2\partial_x\overline{u})
dt'\Big\|_{L_T^{\infty}L_{x}^{2}}\\ 
&  
=:
B_{2,1}+B_{2,2}.
\end{align*}
It is clear that $B_{2,1}$ behaves as $B_1$. 
For the rest of the argument, let's take a compactly supported and smooth function $\chi$,  with the properties that $0 \leq \chi \leq 1$ and $\chi \equiv 1$ on $(-1,1)$. Then, one has
\begin{align*}
B_{2,2}
& \leq
\Big\|\int^t_{0}
e^{-(t-t')(ia \partial^2_x+b\partial^3_x)}
|x|^{m}\chi(x)(d|u|^2 \partial_xu
+eu^2\partial_x \overline{u})\, dt'\Big\|_{L_T^{\infty}L_{x}^{2}} 
\\ 
&  \quad 
+ \Big\|\int^t_{0}
e^{-(t-t')(ia \partial^2_x+b\partial^3_x)}
|x|^{m}(1-\chi(x))(d|u|^2 \partial_xu
+eu^2\partial_x \overline{u})\, dt'\Big\|_{L_T^{\infty}L_{x}^{2}}
\\ &   
=:B_{2,2,1}+B_{2,2,2}.
\end{align*}
By applying Minkowski's integral inequality, Plancherel's identity and exploiting the compact support of $\chi$, we easily get
\begin{align*} 
B_{2,2,1} 
& \lesssim  
\|
|x|^{m}\chi(x)
u^2 \partial_x u
\|_{L^1_TL^2_x}
\lesssim
T^{1/2}
\|u^2 \partial_x u \|_{L^2_TL^2_x},
\end{align*}
and this factor was already controlled in \eqref{eq:B11} above.\\

So far, the parameters $d$ and $e$ have been arbitrary complex numbers. However, for the discussions that follow we found necessary to restrict ourselves to the particular case of $d=2e$, which allows to write part of the nonlinearity as a full derivative, that is,
\begin{equation}\label{eq:fullderiv}
d \, |u|^2 \partial_x u
+ e \, u^2 \partial_x \overline{u}
=
e \partial_x (|u|^2u).
\end{equation}
Therefore, we express
\begin{align*}  
B_{2,2,2}
& \lesssim  
\Big\|\int^t_{0}
e^{-(t-t')(ia \partial^2_x+b\partial^3_x)}
[|x|^{m}(1-\chi(x)),  \partial_x] (|u|^2u)
\, dt' \Big\|_{L_T^{\infty}L_x^2} \nonumber\\ 
&  \quad+
\Big\|\int^t_{0}
e^{-(t-t')(ia \partial^2_x+b\partial^3_x)}
\partial_x
\Big(|x|^{m}(1-\chi(x))|u|^2u\Big)
\, dt' \Big\|_{L_T^{\infty}L_x^2} \nonumber\\
& =: B_{2,2,2,1}+B_{2,2,2,2}.
\end{align*}
Once again, Plancherel's theorem and the commutator estimate   \eqref{eq:commutator} provide
\begin{align*}
B_{2,2,2,1} 
& \lesssim 
\|
[|x|^{m}(1-\chi(x)),  \partial_x] |u|^2u
\|_{L_T^{1}L_x^2}
\lesssim
\|
\partial_x\big(|x|^{m}(1-\chi)\big)
\|_{L_x^{\infty}}\||u|^2u\|_{L_T^1L_x^2}\\
& \lesssim 
T^{1/2}\||u|^2u\|_{L_T^2L_x^2},
\end{align*}
and we can continue as in \eqref{u2u}.
On the other hand, by means of  \eqref{nah3.3} we see that
\begin{align*}
B_{2,2,2,2}
& \lesssim \||x|^{m}(1-\chi)|u|^2u\|_{L_x^1L_T^2}
\lesssim  
\|u\|^2_{L_x^4L_T^{\infty}}
\||x|^{m}u\|_{L_T^2L_x^2} 
\leq 
T^{1/2}
\mu_5^T(u)^2
\||x|^{m}u\|_{L_T^\infty L_x^2}.
\end{align*}
Summing up,
\begin{equation}\label{eq:summaryB}
B
\lesssim 
T^\theta
\Big(
\mu_2^T(u)
\mu_5^T(u)^2
+
\mu_3^T(u)
\mu_4^T(u)
\mu_5^T(u)
+
\mu^T_{1}(u) 
\mu^T_{4}(u)^2
+
\mu_5^T(u)^2
\||x|^{m}u\|_{L_T^\infty L_x^2}
\Big).
\end{equation}
\quad \\
In conclusion, for $u \in X_T^\rho$,
\eqref{eq:goalPHI'},
\eqref{eq:Lsummary},
\eqref{eq:Asummary}
and \eqref{eq:summaryB} lead to
\begin{align*}
\|\Phi(u) \|_{X_T} 
& \leq   
C(1+T)
\|u_0\|_{H^{1/4}}
+
C \||x|^m  u_0\|_{L^2_x}  
+
C T^{\theta}\rho^3
\end{align*}
for some $C, \,\theta>0$. Then, taking
\begin{equation}\label{eq:rho}
\rho:= 2 C (\|u_0\|_{H^{1/4}}+\||x|^m  u_0\|_{L^2_x})
\end{equation}
and   $T>0$ sufficiently small verifying 
\begin{equation}\label{eq:Tsmaill1}
CT\|u_0\|_{H^{1/4}}
+
CT^{\theta}\rho^3 \leq \frac{\rho}{2},
\end{equation}
we deduce \eqref{eq:goalNLSA}.

\quad \\
\underline{\textit{Step 2: $\Phi$ is a contraction}}.
Let $u,v \in X_T^\rho$, for $\rho$ defined in \eqref{eq:rho}. We want to see that
\begin{equation}\label{eq:contraction}
\|\Phi(u)-\Phi(v)\|_{X_T}
\leq K \|u-v\|_{X_T}
\end{equation}
for some $0<K<1$. As in the previous step, it is enough to control the norm 
\begin{equation*}
\||x|^{m}(\Phi(u)-\Phi(v))\|_{L^{\infty}_TL^2_x},
\end{equation*}
where 
\begin{align*}
\Phi(u)-\Phi(v)
& =
\int_0^t 
e^{-(t-t')(ia \partial^2_x+b\partial^3_x)} \\
& \quad \times 
\big( 
c \, \big[|v|^2 v - |u|^2 u \big]
+ d \, \big[|v|^2 \partial_x v - |u|^2 \partial_x u \big]
+ e \, \big[v^2 \partial_x \overline{v} - u^2 \partial_x \overline{u} \big]
\big) \, dt'. 
\end{align*}
Essentially now the task is to convert the differences of the above nonlinear terms into differences of linear factors $u-v$. This can be achieved with the pointwise identities
\begin{align*}
|v|^2 v - |u|^2 u
& \, = \,
(|v|^2+u\overline{v})(v-u)
+
u^2 (\overline{v-u}),\\
|v|^2 \partial_x v - |u|^2 \partial_x u
& \, = \, 
|v|^2 \partial_x(v-u)
+v \partial_x u (\overline{v-u})
+\overline{u} \partial_x u (v-u),\\
v^2 \partial_x \overline{v} - u^2 \partial_x 
\overline{u}
& \, = \, 
v^2 \partial_x(\overline{v-u})
+
(v+u) \partial_x \overline{u}(v-u)
\end{align*}
and repeating the arguments in Step 1. Therefore, it is possible to obtain
$$
\|\Phi(u)- \Phi(v)\|_{X_T}
\lesssim
 T^{\theta} (\|u\|_{X_T} + \|v\|_{X_T})^2 \|u-v\|_{X_T}
\leq 
C T^{\theta}\rho^2\|u-v\|_{X_T},
$$
for some $C, \theta>0$. Finally, choosing $T>0$ satisfying simultaneously \eqref{eq:Tsmaill1} and 
$
C T^{\theta}\rho^2<1,   
$
we conclude \eqref{eq:contraction}.  \qed



\end{document}